\documentclass[letterpaper,12pt,reqno]{amsart}

\usepackage{amsmath, amsthm}
\usepackage{eucal}

\usepackage{amssymb}
\usepackage{color}
\usepackage{ytableau}
\usepackage{latexsym}
\usepackage{graphicx}
\usepackage{amsfonts}
\usepackage{youngtab}

\usepackage{url}

\usepackage{float}		

\usepackage[colorlinks=true, citecolor=blue, linktocpage]{hyperref}

\usepackage{tikz}
\usepackage{bm}		

\usepackage{soul}	

\usepackage[margin=1in]{geometry}

\usepackage{subfigure}	

\newtheorem{theorem}{Theorem}
\newtheorem{lemma}{Lemma}
\newtheorem{corollary}{Corollary}
\newtheorem{conjecture}{Conjecture}

\newtheorem{proposition}{Proposition}

\theoremstyle{definition}
\newtheorem{example}{Example}

\theoremstyle{definition}
\newtheorem{definition}{Definition}

\theoremstyle{definition}
\newtheorem{question}{Question}

\theoremstyle{definition}
\newtheorem{remark}{Remark}

\newcommand{\C}{\mathbb{C}}			
\newcommand{\Grf}{\mathcal{G}}

\newcommand{\col}{\mathrm{col}}
\newcommand{\bigslant}[2]{{\raisebox{.2em}{$#1$}\left/\raisebox{-.2em}{$#2$}\right.}} 

\input xy	
\xyoption{all}

\makeatletter
\newcommand\ackname{Acknowledgements}
\if@titlepage
  \newenvironment{acknowledgements}{%
      \titlepage
      \null\vfil
      \@beginparpenalty\@lowpenalty
      \begin{center}%
        \bfseries \ackname
        \@endparpenalty\@M
      \end{center}}%
     {\par\vfil\null\endtitlepage}
\else
  \newenvironment{acknowledgements}{%
      \if@twocolumn
        \section*{\abstractname}%
      \else
        \small
        \begin{center}%
          {\bfseries \ackname\vspace{-.5em}\vspace{\z@}}%
        \end{center}%
        \quotation
      \fi}
      {\if@twocolumn\else\endquotation\fi}
\fi
\makeatother

\begin{document}

\title[Representation stability of Springer varieties]{Representation stability of the cohomology of Springer varieties and some combinatorial consequences}

\author{Aba Mbirika}
\address{Department of Mathematics, University of Wisconsin-Eau Claire, Eau Claire, WI, U.S.A.}
\email{mbirika@uwec.edu}
\thanks{The first author was partially supported by the UWEC Department of Mathematics and Office of Research and Sponsored Programs.}

\author{Julianna Tymoczko}
\address{Department of Mathematics and Statistics, Smith College, Northampton, MA, U.S.A.}
\email{jtymoczko@smith.edu}
\thanks{The second author was partially supported by a Sloan Research Fellowship as well as NSF grants DMS-1248171 and DMS-1362855.}


\date{\today}

\begin{abstract}
A sequence of $S_n$-representations $\{V_n\}$ is said to be uniformly representation stable if the decomposition of $V_n = \bigoplus_{\mu} c_{\mu,n} V(\mu)_n$ into irreducible representations is independent of $n$ for each $\mu$---that is, the multiplicities $c_{\mu,n}$ are eventually independent of $n$ for each $\mu$.  Church-Ellenberg-Farb proved that the cohomology of flag varieties (the so-called diagonal coinvariant algebra) is uniformly representation stable.  We generalize their result from flag varieties to all Springer fibers.  More precisely, we show that for any increasing subsequence of Young diagrams, the corresponding sequence of Springer representations form a graded co-FI-module of finite type (in the sense of Church-Ellenberg-Farb).  We also explore some combinatorial consequences of this stability.
\end{abstract}

\maketitle



\section{Introduction}

Homological stability is a topological property of certain sequences of topological spaces: given a sequence $\{X_n\}_{n=1}^\infty$ of topological spaces with maps $\phi_n : X_{n} \rightarrow X_{n+1}$ for each $n$, then $\{X_n\}_{n=1}^\infty$ is homologically stable if there exists a positive integer $N$ so that the maps $(\phi_n)_*: H_i(X_{n}) \rightarrow H_i(X_{n+1})$ is an isomorphism whenever $n \geq N$.  In other words the homology groups stabilize after a certain point in this sequence; however the topological spaces change later in the sequence, the changes do not affect the $i^{\mathrm{th}}$ homology group.  

Church and Farb \cite{CF}, and later Church, Ellenberg, and Farb \cite{CEF} defined {\em representation stability} to mimic the topological definition.  Informally, the sequence of $S_n$-representations $V_1 \stackrel{f_1}{\rightarrow} V_2 \stackrel{f_2}{\rightarrow} V_3 \stackrel{f_3}{\rightarrow} \cdots$ is representation-stable if there are $S_n$-equivariant  linear injections $f_n: V_n \rightarrow V_{n+1}$ and if there exists an $N$ so that the multiplicities $c_{\mu,n}$ in the decomposition into irreducibles
\[V_n = \bigoplus_\mu c_{\mu,n} V(\mu)_n\]
are independent of $n$ for all $n \geq N$.  (Section \ref{sec: representation stability} defines representation stability precisely.)  

One important problem is to find families of representations that are representation stable; these families often arise from geometric considerations.  In their original work, Church, Ellenberg, and Farb identify a number of representation stable families arising from geometry/topology and classical representation theory~\cite{CEF}.  Indeed their example of the diagonal coinvariant algebra \cite[Section~5]{CEF} is the sequence of cohomology rings of flag varieties $\{H^*(GL_n/B_n)\}_{n=1}^\infty$ which in some sense is the springboard of this paper.   Since then, others have demonstrated that representation stability arises naturally in many contexts, including arrangements associated to root systems~\cite{Bibby}, linear subspace arrangements~\cite{Gadish}, configuration spaces in $\mathbb{R}^d$~\cite{Hersh-Reiner}, filtrations of Torelli groups~\cite{Patzt}, moduli spaces of Riemann surfaces of genus $g$ with $n$ labeled marked points~\cite{Rolland2011}, and others~\cite{Duque-Rolland,Rolland2016}.

The central goal of this paper is to prove that an important family of $S_n$-representations called {\em Springer representations} is representation stable.  The Springer representation is the archetypal geometric representation: in its most basic form, it arises when the symmetric group $S_n$ acts on the cohomology of a family of subvarieties of the flag variety called {\em Springer fibers}.  The flag variety can be described as the set of nested vector subspaces
$$V_1 \subseteq V_2 \subseteq \cdots \subseteq V_{n-1} \subseteq \mathbb{C}^n$$
where each $V_i$ is $i$-dimensional.  Given a nilpotent matrix $X$, the Springer fiber consists of the flags that are fixed by $X$ in the sense that $XV_i \subseteq V_i$ for all $i$.  Every nilpotent matrix is conjugate to one in Jordan form, which is determined by a partition of $n$ into Jordan blocks.  Since Springer fibers associated to conjugate matrices $X$ and $gXg^{-1}$ are homeomorphic, the Springer fibers are parametrized by partitions of $n$.

Springer first constructed a representation of the symmetric group $S_n$ on the cohomology of Springer fibers \cite{Spr78}.  Since then, the representation has been recreated in many different ways \cite{dCP81,Tani82,BorMac83,Lus84}.  The geometry of Springer fibers encodes key data about representations of the symmetric group.  For instance the top-dimensional cohomology is an irreducible $S_n$-representation \cite{Spr78}; and the ungraded representation on $H^*(Spr_{\lambda})$ is Young's representation associated to the partition $\lambda$ \cite[Introduction]{GarPro92}. 

Our analysis of Springer representations uses the {\bf co-FI} category, which Church, Ellenberg, and Farb defined to concisely describe the compatibility conditions needed for representation stability.   Theorem~\ref{thm:the_main_theorem} proves that for any increasing sequence of Young diagrams, the corresponding Springer representations form a graded co-FI-module of finite type.  The main consequence from our point of view is that sequences of Springer representations are representation stable (see Corollary~\ref{cor:springer_rep_is_representation_stable} for a precise statement).  

From this many other properties follow: (1) for each fixed degree and $n$ large enough, the character is given by a polynomial that is independent of $n$, and in particular (2) the dimension of the Springer representation is eventually polynomial in $n$. It is this second consequence that we explore in Section~\ref{sec:combinatorial_consequences}.

\begin{remark}\label{rem:Hotta}
Curiously, the literature does not make a clear distinction between the Springer representation and its dual. Hotta first observed this and classified existing constructions of the Springer representation up to that point~\cite{Hotta81}. However, the ambiguity persisted with subsequent constructions of the Springer representation.  In this paper, we treat ``the" Springer representation interchangeably with its dual.  We prove that the Garsia-Procesi construction of the Springer representation is a co-FI-module (graded, of finite type) and its {\em dual} is representation stable.  Properties like dimension and decomposition into irreducibles are well-behaved with respect to duality and apply to the original Garsia-Procesi construction, too.
\end{remark}

Kim proves a different kind of stability of Springer representations, giving conditions under which the lower-graded parts of the Springer representation for $\lambda$ coincide with those of $\lambda'$ for partitions $\lambda, \lambda'$ of the same $n$ \cite{Kim_Stability}.  Kim recovers Theorem \ref{thm: at least k+1 rows gives a polynomial dimension formula} using his notion of stability \cite[Corollary 4.3]{Kim_Stability}.  

This paper is organized as follows. In Section~\ref{sec:FI_and_co-FI}, we define the conditions that guarantee the stability of a sequence of $S_n$-representations. In Section~\ref{sec:Springer_theory}, we describe Garsia-Procesi's combinatorial description of the Springer representation, and in particular the so-called {\em Tanisaki ideal}. In Section~\ref{sec:FI-module_structure}, we prove there is no FI-module structure on sequences of Springer representations, except for the trivial representation. On the other hand, in Section~\ref{sec:co-FI-module_structure} we prove that the sequence of Tanisaki ideals forms a co-FI-ideal and deduce that a co-FI-module structure exists on all sequences of Springer representations. In Section~\ref{sec: representation stability}, we conclude with our main result on the stability of the Springer representations. We give some concrete combinatorial consequences of this stability in Section~\ref{sec:combinatorial_consequences}, and a collection of open questions in Section~\ref{sec:open_questions} that probe the new combinatorial ideas raised by representation stability.


\section{{\bf{FI}} and {\bf{co-FI}}}\label{sec:FI_and_co-FI}

In this section we describe FI-modules and co-FI-modules, defined by Church, Ellenberg, and Farb \cite{CEF} to streamline and extend the essential features of Church and Farb's earlier notion of representation  stability \cite{CF}.  We begin with the categories {\bf FI} and {\bf co-FI}, which carry actions of the permutation groups $S_n$ and are constructed to be compatible with inclusions.  The key example of FI- and co-FI-modules for this manuscript is the sequence of polynomial rings $\{k[x_1,\ldots,x_n]\}$.  We then list the properties about FI- and co-FI-modules that we will need to establish that sequences of Springer representations are graded co-FI-modules.

This section provides only what is needed in this paper.  The interested reader is referred to Church, Ellenberg, and Farb's work for many other interesting results \cite{CEF}.

\begin{remark}
We assume $k$ is a field of characteristic zero. Parts of Church-Ellenberg-Farb's theory extends to other fields as well \cite{CEF}.
\end{remark}

\begin{definition}[FI-module, FI-algebra graded FI-algebra]
Let $\mathbf{n}$ denote the set $\{1,\ldots,n\}$.  {\bf FI} is the category whose objects are finite sets and whose morphisms are injections.  This is equivalent to the category whose objects are sets $\mathbf{n} = \{1,2,\ldots,n\}$ and whose morphisms are injections $\mathbf{m} \rightarrow \mathbf{n}$.  An \textit{FI-module} (\textit{FI-algebra}, \textit{graded FI-algebra}) over a commutative ring $k$ is a functor $V$ from {\bf FI} to the category of modules over $k$ ($k$-algebras, graded $k$-algebras).  We usually denote the $k$-module (respectively algebra) $V(\mathbf{n})$ by $V_n$.
\end{definition}

Modules that carry permutation actions compatible with the permutation action on integers provide a rich source of examples of FI-modules.  The next example is the most important for our purposes.

\begin{example}[{\cite[Example 4.1.2 and Remark 4.1.3]{CEF}}]
Define a functor $R$ by:
\begin{itemize}
\item $R$ sends the object $\mathbf{n}$ to the polynomial ring $k[x_1,\ldots,x_n]$ and 
\item $R$ sends the morphism $f: \mathbf{m} \rightarrow \mathbf{n}$ to the homomorphism $$f_*: k[x_1,\ldots,x_m] \rightarrow k[x_1,\ldots,x_n]$$ induced by the condition that $f_*(x_i)=x_{f(i)}$ for all $i \in \mathbf{m}$.
\end{itemize} 
Then $R$ is a graded FI-algebra.  In particular, we have that $R_n$ is the polynomial ring $k[x_1,\ldots,x_n]$.
\end{example}

The category {\bf co-FI} is opposite to {\bf FI}.  The polynomial algebras also form a graded co-FI-algebra, as described below.

\begin{definition}[co-FI-module, co-FI-algebra, graded co-FI-algebra] 
Denote the opposite category of {\bf FI} by {\bf co-FI}.  In particular the objects in {\bf co-FI} are finite sets, without loss of generality the sets $\mathbf{n}$, and the morphisms in {\bf co-FI} from $\mathbf{n}$ to $\mathbf{m}$ are the morphisms in {\bf FI} from $\mathbf{m}$ to $\mathbf{n}$.  A \textit{co-FI-module }(\textit{co-FI-algebra}, \textit{graded co-FI-algebra}) over a commutative ring $k$ is a functor from {\bf co-FI} to the category of $k$-modules ($k$-algebras, graded $k$-algebras).
\end{definition}

For instance if $V$ is an FI-module over $k$ then the dual $V^\vee$ forms a co-FI-module, and vice versa.
 
\begin{example}[{\cite[Example 4.1.2 and Remark 4.1.3]{CEF}}]
Define a functor $R$ as follows:
\begin{itemize}
\item $R$ sends the object $\mathbf{n}$ to the polynomial ring $k[x_1,\ldots,x_n]$ and 
\item $R$ sends the FI-morphism $f: \mathbf{m} \rightarrow \mathbf{n}$ to the homomorphism $$f^*: k[x_1,\ldots,x_n] \rightarrow k[x_1,\ldots,x_m]$$ induced by the condition that 
\[f^*(x_i)=\left\{ \begin{array}{ll}
x_{f^{-1}(i)} &\textup{  if  } i \in \textup{Im}(f) \textup{  and} \\
0 & \textup{  otherwise.}
  \end{array}\right.\]
\end{itemize} 
Then $R$ is a graded co-FI-algebra.
\end{example}

The module categories {\bf FI} and {\bf co-FI} are abelian and so admit many of the algebraic constructions that modules and algebras do.  In particular we can consider co-FI-quotients of a co-FI-module and FI-submodules of an FI-module.

\begin{example}[{\cite[{\bf Classical coinvariant algebra} in Section 5]{CEF}}]\label{example: ideal of symmetric polys is not FI-module}
For each $n$ let $I_n$ denote the ideal of symmetric polynomials with no constant term.  The sequence of ideals $I$ is not an FI-submodule of the FI-module $R$ because the image of an $S_m$-symmetric polynomial under the inclusion map $\iota: \{1,\ldots,m\} \rightarrow \{1,\ldots,n\}$ is not symmetric under the larger group $S_n$.  However the ideals $I$ do form a co-FI-submodule of the co-FI-module $R$.
\end{example}

The next lemma proves that the image of a set of $S_m$-invariant polynomials under an arbitrary injection is the same as the image under the inclusion $\iota$ that is the identity on the integers $\{1,\ldots,m\}$.  We use it to simplify later calculations.

\begin{lemma} \label{lemma: inclusion suffices}
Suppose that $m \leq n$ and that $f: \{1,\ldots,m\} \rightarrow \{1,\ldots,n\}$ is an injection.  Let $\iota$ be the inclusion $\iota: \{1,\ldots,m\} \rightarrow \{1,\ldots,n\}$ that sends $\iota(i) = i$ for each $1 \leq i \leq m$.  

Consider the action of $S_n$ on $R_n$ under which for each $w \in S_n$ and $x_i \in R_n$ we have $w(x_i) = x_{w(i)}$.  Let ${\mathcal{J}} \subseteq R_n$ be any set of polynomials that are preserved under the $S_{n}$-action, in the sense that $w {\mathcal{J}} \subseteq {\mathcal{J}}$ for all permutations $w \in S_{n}$.  Then
\[f^*(\mathcal{J}) = \iota^*(\mathcal{J}).\]
\end{lemma}

\begin{proof}
Our hypothesis means that the polynomial $p \in {\mathcal{J}}$ if and only if the image $w(p) \in {\mathcal{J}}$ for each $w \in S_{n}$.  The map $f: \{1,\ldots,m \} \hookrightarrow \{1,\ldots,n\}$ is an injection so we can define a permutation $w_f \in S_{n}$  by
\[w_f = \left\{ \begin{array}{ll}
f(i) \mapsto i & \textup{ for all $i$ with } 1 \leq i \leq m \\
j_i \mapsto m+i & \textup{ for $j_i$ such that both }  j_1<j_2< \cdots <j_{n-m} \textup{ and } \\
& \hspace{0.5in}  \{ j_1, j_2,  \cdots,j_{n-m}\} \cup \textup{Im}(f) = \{1,2,\ldots,n\}.
\end{array} \right. \]
We know that for each $i$
\[\iota^* (w_fx_i) = \iota^* (x_{w_f(i)})\]
by definition  of the $S_n$-action on permutations.  By construction of $\iota^*$ we have
\[ \iota^* (x_{w_f(i)}) = \left\{ \begin{array}{ll}
0 & \textup{ if } i \in \{ j_1, j_2,  \cdots,j_{n-m}\} \\
x_{f^{-1}(i)} & \textup{ otherwise}.
\end{array} \right. \]
This is exactly $f^*(x_i)$.  Thus $f^*(p) = \iota^* (w_fp)$ for all polynomials $p \in {\mathcal{J}}$ and hence as desired
$f^*({\mathcal{J}}) = \iota^*({\mathcal{J}})$.
\end{proof}

It follows that if $I$ is a sequence of symmetric homogeneous ideals in the graded co-FI-algebra $R$ then we can prove $I$ is a co-FI-submodule simply by considering inclusions---or even just a subset of inclusions.

\begin{corollary}\label{corollary: simple inclusions}
Let $R$ be a graded co-FI-algebra and $I$ be a sequence of ideals $\{I_1,I_2,\ldots\}$ with each $I_n$ an $S_n$-invariant homogeneous ideal in $R_n$.  The following are equivalent:
\begin{enumerate}
\item \label{FI part} $I$ is a co-FI-submodule of $R$.
\item \label{general inclusion part} For each $m<n$ and inclusion $\iota: \{1,\ldots,m\} \rightarrow \{1,\ldots,n\}$ defined by $\iota(i)=i$ for all $i$, the induced map satisfies $\iota^*(I_n)\subseteq I_m$.
\item \label{step inclusion part} For each $n$ and inclusion $\iota_n: \{1,\ldots,n-1\} \rightarrow \{1,\ldots,n\}$, the induced map satisfies $\iota_n^*(I_n)\subseteq I_{n-1}$.
\end{enumerate}
\end{corollary}

\begin{proof}
Part~(\ref{FI part})  is equivalent to Part~(\ref{general inclusion part}) by Lemma~\ref{lemma: inclusion suffices}.  Part~(\ref{general inclusion part}) implies Part~(\ref{step inclusion part}) by definition.  The composition of inclusions $\iota_n \circ \iota_{n-1} \circ \cdots \circ \iota_{m+1}$ is the inclusion $\iota: \{1,\ldots,m\} \hookrightarrow \{1,\ldots,n\}$.  By functoriality if $\iota_{i}^*(I_i) \subseteq I_{i-1}$ for each $i\geq 2$ then $\iota^*(I_n) \subseteq I_m$ for each pair $n \geq m \geq 1$. So Part~(\ref{step inclusion part}) implies Part~(\ref{general inclusion part}).  
\end{proof}

The definition of finitely-generated FI-modules is crucial to representation stability.  It differs importantly from the corresponding definition for modules or rings because it incorporates the underlying $S_n$-action.  As we see in Example~\ref{example: finite generation}, this implies that FI-modules often have fewer generators than we might expect.

In fact, though this does not appear explicitly in the literature, the category of {\em finitely generated} {\bf FI} modules over Noetherian rings is also abelian (by the Noetherian property, proven over $\mathbb{C}$ in~\cite{CEF} and over other Noetherian rings in~\cite{CEFN})~\cite{C20}.  This is the thrust of the arguments that we cite in this paper.

\begin{definition}[Finite generation, finite type]\label{def:finite_generation_and_type}
An FI-module $V$ is \textit{finitely generated} if there is a finite set $S$ of elements in $\coprod_i V_i$ so that no proper sub-FI-module of $V$ contains $S$.  A graded FI-module $V$ has \textit{finite type} if the $i^{\mathrm{th}}$ graded part $V^i$ is finitely generated for each $i$. A graded co-FI-module $W$ is of finite type if its dual $W^\vee$ is a graded FI-module of finite type \cite[Co-FI-algebras in Section 4.2]{CEF}
\end{definition}

\begin{example}[{\cite[{\bf Graded FI-modules of finite type} in Section 4.2]{CEF}, \cite[Example 4.2]{Far14}, \cite[Example 1.4]{Wil14}}]\label{example: finite generation} 
The sequence $R = \{k[x_1,\ldots,x_n]\}$ of polynomial rings are not finitely generated as an FI-module; indeed no ring $k[x_1,\ldots,x_n]$ is a finite dimensional $k$-vector space.  However when graded by polynomial degree, each graded part of the sequence $R$ is finitely generated as an FI-module.  For instance the graded part of degree $3$ is generated by $x_1^3$, $x_1^2x_2$, and $x_1x_2x_3$ since every monomial of degree three is obtained by permuting indices of one of these three.  More generally the graded part of degree $d$ is generated by all monomials of the form $x_1^{d_1}x_2^{d_2}\cdots x_n^{d_n}$ over partitions $d_1 \geq d_2 \geq \cdots \geq d_n \geq 0$ of $d$.  
\end{example}

For completeness, we define the quotient of an FI-algebra or co-FI-algebra.

\begin{definition}[{\cite[Definition~2.76]{CEFa}}]
Let $R$ be a graded FI-algebra.  If $I$ is a graded FI-submodule (respectively co-FI-submodule) for which each object $I_n$ is a homogeneous ideal in $R_n$ then $I$ is called an \textit{FI-ideal} (respectively \textit{co-FI-ideal}).  The quotient FI-module $R/I$ is defined so that $(R/I)_n=R_n/I_n$ for each $n$ (respectively co-FI).
\end{definition}

The following proposition is our main tool, and sketches the main points of \cite[Theorem F]{CEFN}.  This result implies that we only need to prove each sequence of Tanisaki ideals forms a co-FI-submodule; after that, a straightforward algebraic argument allows us allow us to conclude that sequences of Springer representations are representation-stable. 

\begin{proposition}\label{prop: quotient is graded co-FI}
Let $R$ be a graded co-FI-module over a Noetherian ring and let $I$ be a co-FI-ideal.  Then the dual $\left(R/I\right)^\vee$ is a graded FI-module.  If $R^\vee$ has finite type then so does $\left(R/I\right)^\vee$.
\end{proposition}

\begin{proof}
The dual of a graded co-FI-module is a graded FI-module by purely formal properties.  Suppose further that $R^\vee$ has finite type and consider the $j^{\mathrm{th}}$ graded parts $\left(R_n\right)_j^\vee$ and $\left(R_n/I_n\right)_j$ for each $n$.  The dual $\left(R_n/I_n\right)_j^\vee$ is the FI-submodule of $\left(R_n\right)_j^\vee$ consisting of those functionals that vanish on the $j^{\mathrm{th}}$ graded parts $\left(I_n\right)_j = I_n \cap \left(R_n\right)_j$ by definition of graded quotients. Each FI-submodule of an FI-module of finite type over a Noetherian ring is also of finite type \cite[Theorem A]{CEFN}, thus proving the claim.
\end{proof}

In Section~\ref{sec:combinatorial_consequences} and \ref{sec:open_questions} of this manuscript we analyze properties of FI-modules in the case of Springer representations.


\section{Springer theory}\label{sec:Springer_theory}

This section summarizes two key combinatorial tools:  Biagioli-Faridi-Rosas's description of generators for each Tanisaki ideal and Garsia-Procesi's description of a basis for the Springer representation.  

\subsection{The Tanisaki ideal and its generators}  We use Garsia-Procesi's presentation of the cohomology of the Springer fiber, which describes the cohomology as a quotient of a polynomial ring analogous to the Borel construction of the cohomology of the flag variety~\cite{Borel53}.  Let $R_n$ be the polynomial ring $\mathbb{C}[x_1, \ldots, x_n]$.  For each partition $\lambda$ of $n$ Tanisaki defined an ideal $I_{\lambda} \subseteq R_n$ that is now called the {\em Tanisaki ideal}. To describe the Tanisaki ideal, we define certain sets of elementary symmetric functions.  

\begin{definition}
Given a subset $S \subseteq \{x_1, \ldots, x_n\}$ the elementary symmetric function $e_i(S)$ is the polynomial
\[e_i(S) = \sum_{\mbox{\scriptsize $\begin{array}{c}T \textup{ such that }\\ T \subseteq S \textup{ and } |T|=i \end{array}$}} \prod_{x_j \in T} x_j.\]
The set $E_{i,j}^n$ is the set of elementary symmetric functions $e_i(S)$ over all possible subsets $S\subseteq \{x_1, \ldots, x_n\}$ of cardinality $j$.
\end{definition}

\begin{example}
Let $S = \{x_1, x_3, x_4\} \subseteq \{x_1, x_2, \ldots, x_5\}$.  Then $e_2(S) = x_1x_3 + x_1x_4 + x_3x_4$.  Letting $S$ run over all $\binom{5}{3}=10$ size three subsets of $\{x_1, x_2, \ldots, x_5\}$ gives the 10 elementary symmetric polynomials $e_2(S)$ comprising the set $E_{2,3}^5$.
\end{example}

We follow Biagioli, Faridi and Rosas's construction of the Tanisaki ideal~\cite[Definition 3.4]{BFR08}.

\begin{definition}
Let $\lambda = (\lambda_1, \ldots, \lambda_k)$ be a partition of $n$.  The \textit{BFR-filling} of $\lambda$ is constructed as follows.  From the leftmost column of $\lambda$ to the rightmost column, place the numbers $1, 2, \ldots, n-\lambda_1$ bottom to top skipping the top row.  Finally fill the top row from right to left with the remaining numbers $n-\lambda_1+1, \ldots, n$.   The \textit{BFR-generators} are polynomials in the set $\Grf(\lambda)$ defined as the following union:  If the box filled with $i$ has $j$ in the top row of its column, then include the elements of $E_{i,j}^n$ in the set $\Grf(\lambda)$.
\end{definition}

\begin{example}
Consider the partition $\lambda = (4,2,2,1)$.  Then the BFR-filling of $\lambda$ is
$$
\ytableausetup{centertableaux,boxsize=1.25em}
\begin{ytableau}
9& 8 & 7 & 6 \\
3 & 5 \\
2 & 4 \\
1
\end{ytableau}
$$
and hence the set $\Grf(\lambda)$ is
$$ \Grf(\lambda) = \bigcup_{i \in \{1,2,3,9\}} \!\!\! E^9_{i,9} \; \cup \; \bigcup_{i \in \{4,5,8\}} E^9_{i,8} \; \cup \;  E^9_{7,7} \; \cup \; E^9_{6,6}.$$
The first union of sets arises from the first column, the second from the second, and so on.  Each set of the form $E^9_{i,9}$  contains the single elementary symmetric function $e_i(x_1, x_2, \ldots, x_9)$.

\end{example}

\begin{theorem}[Biagioli-Faridi-Rosas~{\cite[Corollary 3.11]{BFR08}}]
Given a partition $\lambda$ of $n$ the Tanisaki ideal $I_\lambda$ is generated by the set $\Grf(\lambda)$.
\end{theorem}

\subsection{Combinatorial presentation for the Springer representation and the Garsia-Procesi basis}
 It turns out that the natural $S_n$-action on $R_n$ given by $$w \cdot p(x_1, \ldots, x_n) = p(x_{w(1)}, \ldots, x_{w(n)})$$ restricts to the Tanisaki ideal $I_{\lambda}$.  Thus the $S_n$-action on $R_n$ induces an $S_n$-action on $R_n/I_{\lambda}$.  The key result of Garsia-Procesi's work (with others \cite{Kra80,dCP81,Tani82}) is that this quotient is the cohomology of the Springer variety.

\begin{proposition}[Garsia-Procesi {\cite{GarPro92}}] Let $R_n = \mathbb{C}[x_1, \ldots, x_n]$ and let $I_\lambda$ be the Tanisaki ideal corresponding to the partition $\lambda$.  The quotient $R_n/I_{\lambda}$ is isomorphic to the cohomology of the Springer variety $H^*(Spr_{\lambda})$ as a graded $S_n$-representation.
\end{proposition}

Moreover Garsia-Procesi describe an algorithm to compute a nice basis $\mathcal{B}(\lambda)$ of monomials for $R_n/I_{\lambda}$.  Our exposition owes much to the presentation in the first author's work \cite[Definition 2.3.2]{Mbir10}.

\begin{definition}[Garsia-Procesi {\cite[Section 1]{GarPro92}}]
If $\lambda = (1)$ is the unique partition of $1$ then $\mathcal{B}(\lambda) = \{1\}$.  If $n>1$ and $\lambda$ is a partition of $n$ with $k$ parts then:
\begin{itemize}
\item Number the rightmost box in the $i^{\mathrm{th}}$ row of $\lambda$ with $i$ for each $i \in \{1,\ldots,k\}$.
\item For each $i \in \{1,\ldots,k\}$ construct the partition $\lambda_i$ of $n-1$ by erasing the box labeled $i$ and rearranging rows if needed to obtain a Young diagram once again.
\item Recursively define $\mathcal{B}(\lambda)$ as \[ \mathcal{B}(\lambda) = \bigcup_{i=1}^k x_n^{i-1} \hspace{0.25em} \mathcal{B}(\lambda_i).\]
\end{itemize}
From the \textit{GP-algorithm} above we construct the \textit{GP-tree} as follows.  Let $\lambda$ sit alone at Level $n$ in a rooted tree directed down.  Since $\lambda$ has $k$ parts, create $k$ edges labeled $x_n^0, x_n^1, \ldots, x_n^{k-1}$ left-to-right to the $k$ subdiagrams $\lambda_i$ for each $i \in \{1, \ldots, k\}$ in Level $n-1$.  For each of these subdiagrams and their descendants, recursively repeat this process.  The process ends at Level $1$, whose diagrams all contain one single box.  Multiplying the edge labels on any downward path gives a unique \textit{GP-monomial} in the \textit{GP-basis}.
\end{definition}

\begin{example}\label{example: GP-tree}
We illustrate the first two steps of the GP-algorithm on the partition  of 5 given by $\lambda = (2,2,1)$.  Number the far-right boxes and branch down from Level 5 to Level 4 of the recursion as follows:
$$
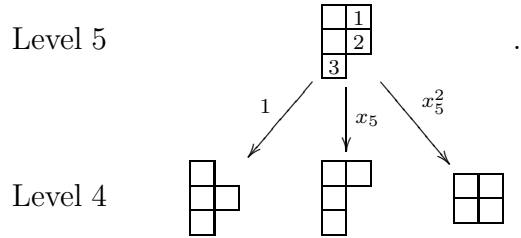
\xymatrix{
\mathrm{Level \; 5} & & {\ytableausetup{centertableaux,boxsize=.75em}
\begin{ytableau} {} & {\mbox{\tiny 1}} \\ {} & {\mbox{\tiny 2}} \\ {\mbox{\tiny 3}} \end{ytableau}} \ar[dl]_{1} \ar[d]^{x_5} \ar[dr]^{x_5^2} \\
\mathrm{Level \; 4} & {\ytableausetup{centertableaux,boxsize=.75em}
\begin{ytableau} {} \\ {} & {} \\ {} \end{ytableau}} & {\ytableausetup{centertableaux,boxsize=.75em}
\begin{ytableau} {} & {} \\ {} \\ {} \end{ytableau}} & {\ytableausetup{centertableaux,boxsize=.75em}
\begin{ytableau} {} & {} \\ {} & {} \end{ytableau}}
}.
$$
After rearranging rows to obtain a Young diagram, we begin the recursion again on each of the three partitions to produce Level 3 of the tree.
%
%
$$
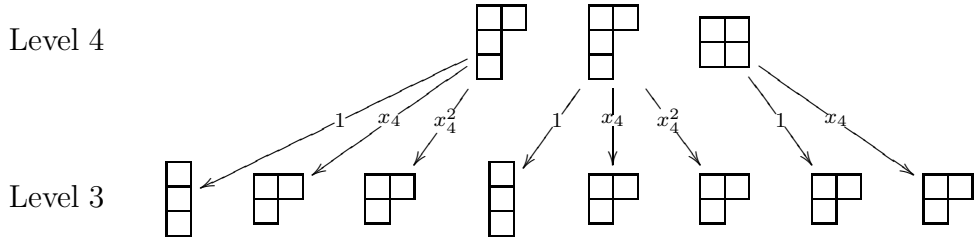
\xymatrixcolsep{.225in}
\xymatrix{
\mathrm{Level \; 4} & & & & {\ytableausetup{centertableaux,boxsize=.75em}
\begin{ytableau} {} & {} \\ {} \\ {} \end{ytableau}} \ar[dlll]|{1} \ar[dll]|{x_4} \ar[dl]|{x_4^2} & {\ytableausetup{centertableaux,boxsize=.75em}
\begin{ytableau} {} & {} \\ {} \\ {} \end{ytableau}} \ar[dl]|{1} \ar[d]|{x_4} \ar[dr]|{x_4^2} & {\ytableausetup{centertableaux,boxsize=.75em}
\begin{ytableau} {} & {} \\ {} & {} \end{ytableau}} \ar[dr]|{1} \ar[drr]|{x_4} \\
\mathrm{Level \; 3} & {\ytableausetup{centertableaux,boxsize=.75em}
\begin{ytableau} {} \\ {} \\ {} \end{ytableau}} & {\ytableausetup{centertableaux,boxsize=.75em}
\begin{ytableau} {} & {} \\ {} \end{ytableau}} & {\ytableausetup{centertableaux,boxsize=.75em}
\begin{ytableau} {} & {} \\ {} \end{ytableau}} & {\ytableausetup{centertableaux,boxsize=.75em}
\begin{ytableau} {} \\ {} \\ {} \end{ytableau}} & {\ytableausetup{centertableaux,boxsize=.75em}
\begin{ytableau} {} & {} \\ {} \end{ytableau}} & {\ytableausetup{centertableaux,boxsize=.75em}
\begin{ytableau} {} & {} \\ {} \end{ytableau}} & {\ytableausetup{centertableaux,boxsize=.75em}
\begin{ytableau} {} & {} \\ {} \end{ytableau}} & {\ytableausetup{centertableaux,boxsize=.75em}
\begin{ytableau} {} & {} \\ {} \end{ytableau}}
}
$$
By Level 1, there will be 30 diagrams, each equal to the partition $(1)$. Recovering the Garsia-Procesi basis from this tree is equivalent to multiplying the edge labels of the 30 paths.  The reader can verify that we obtain the following basis $\mathcal{B}(\lambda)$:
\begin{center}
\begin{tabular}{c|c|l}
degree & \# & \hspace{1in} monomials in $\mathcal{B}(\lambda)$ \\ \hline \hline
0 & 1 & 1 \\ \hline
1 & 4 & $x_i$ for $2 \leq i \leq 5$ \\ \hline
2 & 9 & $x_i^2$ for $3 \leq i \leq 5$ and $x_ix_j$ for $2 \leq i<j \leq 5$\\ \hline
3 & 11 & $x_i x_j x_5$ for $2 \leq i<j \leq 4$,\\
  &    & $x_i x_j^2$ for $2 \leq i < j \leq 5$, and\\
  &    & $x_i^2 x_5$ for $3 \leq i \leq 4$\\ \hline
4 & 5 & $x_i x_j^2 x_5$ for $2 \leq i<j \leq 4$ and $x_i x_4 x_5^2$ for $2 \leq i \leq 3$\\
\hline
\end{tabular}
\end{center}
\end{example}

Garsia-Procesi bases have several nice containment properties that we use when analyzing the FI- and co-FI-structure of Springer representations.  The first describes the relationship between dominance order and the Garsia-Procesi bases.

\begin{proposition}[Garsia-Procesi {\cite[Proposition 4.1]{GarPro92}}] \label{proposition: dominance order}
Let $\lambda$ and $\lambda'$ be partitions of $n$ and suppose that $\lambda \unlhd \lambda'$ in dominance order, namely that we have $\lambda_1+\lambda_2+\cdots+\lambda_i \leq \lambda'_1+\lambda'_2+\cdots+\lambda'_i$ for all $i\geq 1$.  Then $\mathcal{B}(\lambda') \subseteq \mathcal{B}(\lambda)$.
\end{proposition}

The second describes the relationship between containment of Young diagrams and the Garsia-Procesi bases.

\begin{lemma}\label{lemma: containment and Garsia-Procesi basis}
If $\lambda \subseteq \lambda'$ then $\mathcal{B}(\lambda) \subseteq \mathcal{B}(\lambda')$.
\end{lemma}

\begin{proof}
We prove the claim assuming that $\lambda'$ has exactly one more box than $\lambda$.  Repeating the argument gives the desired result.

Consider the subdiagrams $\lambda_1, \lambda_2, \ldots, \lambda_k$ obtained from $\lambda'$ in the recursive definition of the Garsia-Procesi algorithm.  The Garsia-Procesi algorithm says 
\[\mathcal{B}(\lambda') = \bigcup_{i=1}^k x_n^{i-1} \hspace{0.25em} \mathcal{B}(\lambda_i)  = \mathcal{B}(\lambda_1) \cup \bigcup_{i=2}^k x_n^{i-1} \hspace{0.25em} \mathcal{B}(\lambda_i) \]
so $\mathcal{B}(\lambda') \supseteq \mathcal{B}(\lambda_1)$.  By construction $\lambda_i$ is obtained from $\lambda'$ by removing a box that is above and possibly to the right of the box removed for $\lambda_{i+1}$ for each $i \in \{1,\ldots,k-1\}$.  In particular $\lambda_i \unlhd \lambda_{i+1}$ in dominance order.  Proposition \ref{proposition: dominance order} implies that $\mathcal{B}(\lambda_{i+1}) \subseteq \mathcal{B}(\lambda_i)$ for each $i$ and so $\mathcal{B}(\lambda') \supseteq \mathcal{B}(\lambda_i)$ for each $i$.  Since $\lambda$ is obtained from $\lambda'$ by removing a single box, we know $\lambda = \lambda_i$ for some $i$.  The claim follows.
\end{proof}


\section{The Springer representations with FI-module structure}\label{sec:FI-module_structure}

Recall that the sequence of polynomial rings carries an FI-module structure and that both the Tanisaki ideal and the quotients $R_n/I_{\lambda}$ carry an $S_n$-action.  The question in this section is: do these fit together to give an FI-module structure on Springer representations?  Lemma \ref{lemma: containment and Garsia-Procesi basis} suggests that the answer could be yes, since it proved that if $\lambda \subseteq \lambda'$ then the Garsia-Procesi basis for $R_n/I_{\lambda}$ is contained in the Garsia-Procesi basis for $R_{n'}/I_{\lambda'}$.

We prove that this is misleading: the inclusion $R_n/I_{\lambda} \hookrightarrow R_{n'}/I_{\lambda'}$ in no way preserves the $S_n$ action (and is not what Church and Farb call a consistent sequence~\cite[pg.6]{CF}).  In particular we prove that there is {\em no} FI-module structure on sequences of Springer representations, except for the trivial representation.  Church, Ellenberg, and Farb observed that the sequence of ideals of symmetric functions with no constant term is not an FI-ideal (see Example~\ref{example: ideal of symmetric polys is not FI-module}); in our language, they study the special case of the sequence of Tanisaki ideals  $\{I_{\lambda_1}, I_{\lambda_2}, \ldots\}$ where each $\lambda_i$ is a column with $i$ boxes.

\begin{theorem}
For each positive integer $n$, let $\lambda_n$ denote a Young diagram with $n$ boxes.
The only sequence of Young diagrams $\lambda_1 \subseteq \lambda_2 \subseteq \lambda_3 \subseteq \cdots$ for which the sequence of Tanisaki ideals $\{I_{\lambda_1}, I_{\lambda_2}, I_{\lambda_3}, \ldots\}$ forms an FI-ideal is {\tiny \Yvcentermath1 $\yng(1) \rightarrow \yng(2) \rightarrow \yng(3) \rightarrow \cdots$}.
\end{theorem}

\begin{proof}
Suppose that $\lambda$ is a partition of $n-1$ and $\lambda'$ is a partition of $n$ with $\lambda \subseteq \lambda'$.  We prove that if $f: \{1,\ldots,n-1\} \rightarrow \{1,\ldots,n\}$ is an injection for which $f_*(I_\lambda) \subseteq I_{\lambda'}$ then $I_{\lambda'} = \langle x_1, x_2, \ldots, x_n\rangle$.  By definition of FI-modules this suffices to prove our claim.

We first show that $I_{\lambda'}$ contains one of the variables $x_1, \ldots, x_n$.  Construct the BFR-fillings of both $\lambda$ and $\lambda'$ and compare the boxes labeled $1$.  A Young diagram with at least two rows has label $1$ in the bottom box of the first column.  A Young diagram with only one row has label $1$ in the rightmost box of its row.   If either $\lambda$ or $\lambda'$ has just one row then the generating set for its Tanisaki ideal contains $e_1(x_1)$.  Hence $I_{\lambda'} \supseteq I_{\lambda}$ contains the variable $x_1$.  Otherwise the top-left boxes of the BFR-fillings for $\lambda$ and $\lambda'$ have the labels $n-1$ and $n$ respectively, so $\Grf(\lambda)$ contains $e_1(x_1, \ldots, x_{n-1})$ and $\Grf(\lambda')$ contains $e_1(x_1,\ldots,x_n)$.  In this case 
\[e_1(x_1,\ldots,x_n) - f_*(e_1(x_1,\ldots, x_{n-1})) = x_j\] 
lies in $I_{\lambda'}$ for the unique $j \in \{1,2,\ldots,n\} - \textup{Im}(f)$.  

Thus $I_{\lambda'}$ contains at least one of the variables $x_1, \ldots, x_n$.  The Tanisaki ideal is invariant under the action of $S_n$ that permutes the variables $x_1, \ldots, x_n$ so in fact $I_{\lambda'} = \left\langle x_1, \ldots, x_n \right\rangle$.  

We conclude that
$$ R_n/I_{\lambda'} = R_n/\left\langle x_1, \ldots, x_n \right\rangle \cong \C$$
is trivial.  Since $\lambda'$ gives the trivial representation, it consists of a single row, as desired.
\end{proof}


\section{The Springer representation with the co-FI-module structure}\label{sec:co-FI-module_structure}

In this section we prove that Springer representations admit a co-FI-module structure.  Recall from Section~\ref{sec:FI_and_co-FI} that the sequence of polynomial rings forms a co-FI-module.  This section shows that the natural restriction maps from $k[x_1, \ldots, x_n]$ to $k[x_1, \ldots, x_{n-1}]$ are defined on Tanisaki ideals, too.  Our proof mimics a similar proof for the cohomology of flag varieties in the arXiv version of a paper by Church, Ellenberg, and Farb~\cite[Theorem~3.4]{CEFa}. A surprising feature of our result is that this co-FI-module structure exists for Springer representations corresponding to {\em every possible} sequence $\lambda_1 \subseteq \lambda_2 \subseteq \lambda_3 \subseteq \cdots$ of Young diagrams.

The main features of the proof were outlined in previous sections, especially Section~\ref{sec:FI_and_co-FI}, which collected steps that reduce the proof that a sequence of Springer representations form a co-FI-module to proving that particular inclusions $\iota_n^*: R_n \rightarrow R_{n-1}$ preserve Tanisaki ideals.  We prove the main theorem first and then prove the key lemma.

\begin{theorem}\label{thm:the_main_theorem}
Suppose $\lambda_1 \subseteq \lambda_2 \subseteq \cdots$ is a sequence of Young diagrams for which $\lambda_n$ has $n$ boxes for each $n$.  Then the sequence $\{R_n/I_{\lambda_n}\}$ forms a graded co-FI-module of finite type. 
\end{theorem}

\begin{proof}
Definition~\ref{def:finite_generation_and_type} states that to show $\{R_n/I_{\lambda_n}\}$ is a graded co-FI-module of finite type, we must show 1) it is a graded co-FI-module and 2) its dual is a graded FI-module of finite type.

Assuming a field $k$ of characteristic zero, Proposition~\ref{prop: quotient is graded co-FI} states that if $R$ is a graded co-FI-module and $I$ is a co-FI-ideal then $R/I$ is a graded co-FI-module.  The sequence of polynomial rings is a graded co-FI-algebra under the co-FI-algebra structure that sends the map $f: \{1,\ldots,m\} \rightarrow \{1,\ldots,n\}$ to the map with $f^*(x_i)$ equal to $x_{f^{-1}(i)}$ if $i \leq m$ and zero otherwise (see Example~\ref{example: finite generation}).  Corollary~\ref{corollary: simple inclusions} proves that if $\{\lambda_1 \subseteq \lambda_2 \subseteq \cdots \}$ is a sequence of Young diagrams for which $\iota_{n}^*(I_{\lambda_n}) \subseteq I_{\lambda_{n-1}}$ for every $n \geq 2$ then the sequence of Tanisaki ideals $\{I_{\lambda_1}, I_{\lambda_2}, \ldots\}$ forms a co-FI-ideal.  Lemma~\ref{lemma: standard inclusion acts right on Tanisaki ideals} proves that $\iota_{n}^*(I_{\lambda_n}) \subseteq I_{\lambda_{n-1}}$ for every $n \geq 2$ and for every sequence of Young diagrams  $\{\lambda_1 \subseteq \lambda_2 \subseteq \cdots \}$.   Thus $\{R_n/I_{\lambda_n}\}$ is a graded co-FI module.  

It is known that the dual ${ (R_n)^\vee }$ is a graded FI-module of finite type.  The proof uses the facts that 0) the variables $x_1, \ldots, x_n \in R_n$ form a basis for the dual to $k^n$, 1) the $j^{\mathrm{th}}$ graded part $(R_n)^\vee_j$ is isomorphic to the $j^{\mathrm{th}}$ symmetric power of $k^n$, and 2) the $k$-modules $k, k^2, k^3, \ldots, k^n, \ldots$ are the parts of a natural FI-module $M(1)$ whose algebraic properties are well understood; for details, see~\cite[proof of Theorem F]{CEFN}.

Proposition~\ref{prop: quotient is graded co-FI} then implies that $\{\left(R_n/I_{\lambda_n}\right)^\vee\}$ is a graded FI-module of finite type, proving the claim.
\end{proof}

\subsection{Sequences of Tanisaki ideals form a co-FI-ideal}
We now analyze the BFR-generators  of the Tanisaki ideals to show that $\iota_n^*(I_{\lambda'}) \subseteq I_{\lambda}$ when $\lambda \subseteq \lambda'$.  In an interesting twist, this co-FI-module structure exists for any path in the poset of Young diagrams.  Our proof that these sequences of Tanisaki ideals $\{I_{\lambda_i}\}$ form a co-FI-ideal occurs over the course of the following lemmas.  Each result actually analyzes a subset of homogeneous polynomials of the same degree inside $I_{\lambda}$ so our proofs preserve grading (though we don't use this fact).

We note the following well-known relation (e.g., \cite[pg.~21]{Pro07}, \cite[Equation (3.1)]{FGP97}). 

\begin{proposition}
Let $S \subseteq \{x_1,\ldots,x_{n}\}$.  If $x_{n} \in S$ then $e_i(S)$ decomposes as
$$e_i(S) = x_{n} \cdot e_{i-1}(S-\{x_{n}\}) + e_i(S-\{x_{n}\}).$$
\end{proposition}

This immediately implies the following.  (Recall that $i$ is the degree of each function in the set $E_{i,j}^n$ while $j$ is the cardinality $|S|$ of the variable subset $S \subseteq \{x_1,x_2,\ldots,x_n\}$.)

\begin{lemma}\label{lemma: image of BFR-generators for a box}
\[\iota_n^*\left(E_{i,j+1}^{n}\right) = E_{i,j+1}^{n-1} \cup E_{i,j}^{n-1}\]
where $E_{i,k}^{n-1}$ is empty if $k > n-1$ or if $i > k$ or if $n-1 < 1$.
\end{lemma}

The next lemma is our main tool: if an ideal contains $E_{i,j}^n$ then it contains $E_{i,j+1}^n$ as well.

\begin{lemma}\label{lemma: E_ij contains E_ij+1}
The $\mathbb{Q}$-linear span of $E_{i,j}^n$ contains $E_{i,j+1}^n$.
\end{lemma}

\begin{proof}
Let $S' \subseteq \{x_1, \ldots, x_n\}$ be an arbitrary subset of cardinality $j+1$.  We construct $e_i(S')$ explicitly in the $\mathbb{Q}$-linear span of $E_{i,j}^n$.  Consider the polynomial
\[p =  \! \!  \! \!  \! \!  \! \!  \sum_{\mbox{\scriptsize $\begin{array}{c}S \textup{ such that } \\ S \subseteq S' \textup{ and } |S|=j \end{array}$}} \! \!  \! \!  \! \!  \! \! e_i(S)\]
which is in $\langle E_{i,j}^n \rangle_{\mathbb{Q}}$ by definition.  Let $T$ be a subset of $S'$ of cardinality $i$ and consider the coefficient of the monomial $\prod_{x_k \in T} x_k$ in $p$. Whenever $S \subseteq S'$ contains $T$ the monomial $\prod_{x_k \in T} x_k$ appears with coefficient $1$ in $e_i(S)$.  There are $j+1$ subsets of $S'$ with cardinality $j$ and all but $i$ of them contain $T$.  Hence the coefficient of $\prod_{x_k \in T} x_k$ in $p$ is exactly $j+1-i$.  Thus the polynomial $\frac{1}{j+1-i}p = e_i(S')$ as desired.
\end{proof}

\begin{example}
Consider the sets $E_{2,3}^5$ and $E_{2,4}^5$ and let $S' = \{x_1,x_2,x_4,x_5\}$.  The polynomial $p$ from Lemma \ref{lemma: E_ij contains E_ij+1} is
\begin{align*}
p = (x_1x_2+x_1x_4+x_2x_4) &+ (x_1x_2+x_1x_5+x_2x_5)\\
    &\hspace{-.75in} +(x_1x_4+x_1x_5+x_4x_5)+(x_2x_4+x_2x_5+x_4x_5)
\end{align*}
which simplifies as desired to
\[p = 2(x_1x_2+x_1x_4+x_1x_5+x_2x_4+x_2x_5+x_4x_5)=(4-2)e_2(S').\]
\end{example}

Applying Lemma~\ref{lemma: E_ij contains E_ij+1} repeatedly gives the following.

\begin{corollary}\label{corollary: does the trick}
If an ideal $I \subseteq \mathbb{Q}[x_1, \ldots, x_n]$ contains the subset $E_{i,j}^n$ then it contains $E_{i,j+k}^n$ for all $k \in \{1, 2, \ldots, n-j\}$.
\end{corollary}

We now use these combinatorial properties of symmetric functions together with the BFR-generators of the Tanisaki ideal to prove the main lemma of this section

\begin{lemma} \label{lemma: standard inclusion acts right on Tanisaki ideals}
Suppose that $\lambda$ is a Young diagram with $n-1$ boxes.  If $\lambda'$ is obtained from $\lambda$ by adding one box then $\iota_n^*(I_{\lambda'}) \subseteq I_{\lambda}$.
\end{lemma}

\begin{proof}
Below we give schematics for the relative configurations of $\lambda$, $\lambda'$, and the deleted box (shown in light grey). 
$$
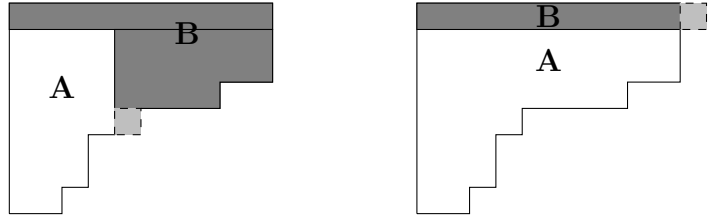
\begin{tikzpicture}[scale=.35]
\draw (0,0)--(0,8)--(10,8)--(10,7)--(10,5)--(8,5)--(8,4)--(4,4)--(4,3)--(3,3)--(3,1)--(2,1)--(2,0)--(0,0);
\draw [fill=gray] (0,7) rectangle (10,8);
\draw [fill=gray] (4,4)--(4,7)--(10,7)--(10,5)--(8,5)--(8,4)--(4,4);
\node [above] at (2,4) {\textbf{A}};
\node [above] at (6.75,6.05) {\textbf{B}};
\draw [white] (4,3)--(4,4)--(5,4);\draw [fill=lightgray, dashed] (4,3)--(4,4)--(5,4)--(5,3)--(4,3);
\end{tikzpicture}
\hspace{.75in}
\begin{tikzpicture}[scale=.35]
\draw (0,0)--(0,8)--(10,8)--(10,7)--(10,5)--(8,5)--(8,4)--(4,4)--(4,3)--(3,3)--(3,1)--(2,1)--(2,0)--(0,0);
\node [above] at (5,5) {\textbf{A}};
\draw [fill=gray] (0,7) rectangle (10,8);
\node [above] at (5,6.675) {\textbf{B}};
\draw [white] (10,7)--(10,8);\draw [fill=lightgray, dashed] (10,7)--(10,8)--(11,8)--(11,7)--(10,7);
\end{tikzpicture}
$$
We will compare the BFR-generators for $I_{\lambda'}$ and $I_{\lambda}$ for region A, region B, and the two different cases of grey boxes.  

Each box in region A corresponds to BFR-generators $E_{i,j}^{n-1}$ in $I_{\lambda}$ and $E_{i,j+1}^{n}$ in $I_{\lambda'}$.  In each of these cases the image 
\[\iota_n^*(E_{i,j+1}^{n}) = E_{i,j+1}^{n-1} \cup E_{i,j}^{n-1}\] 
by Lemma \ref{lemma: image of BFR-generators for a box}.  By construction $I_{\lambda}$ contains $E_{i,j}^{n-1}$ so by Corollary \ref{corollary: does the trick} we know $I_{\lambda}$ contains $E_{i,j+1}^{n-1}$ as well.  It follows that $I_{\lambda}$ contains $\iota_n^*(E_{i,j+1}^{n})$ for every box in region A.

If the box in $\lambda' - \lambda$  is not on the first row, then it is labeled $E_{i,j+1}^{n}$ in $I_{\lambda'}$ and the box above it is labeled $E_{i,j}^{n-1}$ in $I_{\lambda}$. This is the case of the previous paragraph. 

If the box in $\lambda' - \lambda$ is on the first row, then $I_{\lambda'}$ contains $E_{i,i}^n$ while $I_{\lambda}$ contains $E_{i,i}^{n-1}$.  Lemma~\ref{lemma: image of BFR-generators for a box} shows that $\iota_n^*(E_{i,i}^{n}) = E_{i,i}^{n-1} \cup E_{i,i-1}^{n-1}$.  However $E_{i,i-1}^{n-1}$ is empty by definition.  So $I_{\lambda}$ contains $\iota^*(E_{i,i}^{n})$ in this case too.

Finally, each box in region B corresponds to BFR-generators $E_{i,j}^{n-1}$ in $I_{\lambda}$ and $E_{i+1,j+1}^{n}$ in $I_{\lambda'}$. As above we know
\[\iota_n^*(E_{i+1,j+1}^{n}) = E_{i+1,j+1}^{n-1} \cup E_{i+1,j}^{n-1}\] 
by Lemma \ref{lemma: image of BFR-generators for a box}.  The box labeled $i+1$ in $\lambda$ is either above but in the same column as in $\lambda'$ or in a column to the right of $i+1$ in $\lambda'$. In either case $E_{i+1,j-k}^{n-1}$ is a BFR-generator for $\lambda$ for some nonnegative integer $k$.  Corollary \ref{corollary: does the trick} implies that both $E_{i+1,j+1}^{n-1}$ and  $E_{i+1,j}^{n-1}$ are in $I_{\lambda}$.  Hence $\iota_n^*(E_{i+1,j+1}^{n})$ is in $I_{\lambda}$ for every box in region B, proving the claim.
\end{proof}


\section{Representation stability of Springer representations}\label{sec: representation stability}

We want to say that a sequence of representations $\{V_n\}$ is representation stable if, when decomposed into irreducible representations, the sequence of multiplicities of each irreducible representation eventually becomes constant. This doesn't quite make sense because the irreducible representations of the symmetric group $S_n$ depend on $n$, and in fact correspond to the partitions of $n$.  The next definition describes a particular family of irreducible representations whose multiplicities we use to define representation stability.

\begin{definition}[Irreducible $S_n$-representation $V(\mu)_n$~{\cite[Section~2.1]{CF}}]
Let $\mu = (\mu_1, \ldots, \mu_l)$ be a partition of $k$.  For any $n \geq k + \mu_1$ define the \textit{padded partition} to be
$$ \mu[n] := (n-k, \mu_1, \ldots, \mu_l).$$
Define $V(\mu)_n$ to be the irreducible $S_n$-representation
$$V(\mu)_n := V_{\mu[n]}.$$
\end{definition}

Note that every partition of $n$ can be written as $\mu[n]$ for a unique partition $\mu$ and hence every irreducible $S_n$-representation is of the form $V(\mu)_n$ for a unique partition $\mu$.

We can now define representation stability precisely.

\begin{definition}
A sequence of $S_n$-representations $V_1 \stackrel{f_1}{\rightarrow} V_2 \stackrel{f_2}{\rightarrow} V_3 \stackrel{f_3}{\rightarrow} \cdots$ is {\em representation stable} if 
\begin{enumerate}
\item the linear maps $f_n: V_n \rightarrow V_{n+1}$ are $S_n$-equivariant, in the sense that for each $w \in S_n$ we have $f_n \circ w = w(n+1) \circ f_n$ where $w(n+1) \in S_{n+1}$ is the permutation that sends $n+1 \mapsto n+1$ while otherwise acting as $w$;
\item the maps $f_n: V_n \rightarrow V_{n+1}$ are injective;
\item the span of the $S_{n+1}$-orbit of the image $f_n(V_n)$ is all of $V_{n+1}$; and
\item if $V_n$ decomposes into irreducible representations as 
\[V_n = \bigoplus_\mu c_{\mu,n} V(\mu)_n\] 
then there exists $N$ so that the multiplicities $c_{\mu,n}$ are independent of $n$ for all $n \geq N$. 
\end{enumerate}
If $N$ is independent of $\mu$ then $V_1 \stackrel{f_1}{\rightarrow} V_2 \stackrel{f_2}{\rightarrow} V_3 \stackrel{f_3}{\rightarrow} \cdots$ is {\em uniformly} representation stable. 
\end{definition}

FI-module structure corresponds to a sequence of representations being uniformly representation stable.  Moreover the stability of the multiplicities $c_{\mu,n}$ corresponds to a kind of stabilization of the characters of the representations, and thus the dimensions of the representations.  

Since $\{R_n/I_{\lambda_n}\}$ forms a co-FI-module, we will actually show that the dual $\{\left(R_n/I_{\lambda_n}\right)^\vee\}$ is representation stable.  The term {\em Springer representation} is applied interchangeably to these two dual representations (as described in the Introduction--see Remark~\ref{rem:Hotta}).  Moreover, the key implications of representation stability (including dimensions and characters) apply to $\{R_n/I_{\lambda_n}\}$ by virtue of applying to its dual.

More precisely we have the following.

\begin{corollary}\label{cor:springer_rep_is_representation_stable}
Suppose $\lambda_1 \subseteq \lambda_2 \subseteq \lambda_3 \subseteq \cdots$ is a sequence of Young diagrams for which $\lambda_n$ has $n$ boxes for each $n$.  Then the following are all true:
\begin{enumerate}
\item Each graded part of the sequence $\{\left(R_n/I_{\lambda_n}\right)^\vee\}$ is uniformly representation stable.
\item For each $k$ and each $n$ let $\chi_{k,n}$ be the character in the $k^{\mathrm{th}}$ graded part of $R_n/I_{\lambda_n}$.  The sequence $(\chi_{k,1},\chi_{k,2},\chi_{k,3},...)$ is eventually polynomial in the sense of \cite{CEF}.
\item For each $k$ and $n$ let $d_{k,n}$ be the dimension of the $k^{\mathrm{th}}$ graded piece of $R_n/I_{\lambda_n}$ as a complex vector space.  Then the sequence $(d_{k,1},d_{k,2},d_{k,3},\ldots)$ is eventually polynomial, in the sense that there is an integer $s_k$ and polynomial $p_k(n)$ in $n$ so that for all $n \geq s_k$ the dimensions $d_{k,n}=p_k(n)$.
\end{enumerate}
\end{corollary}

\begin{proof}
This follows immediately from corresponding results of Church, Ellenberg, and Farb. 
Part (1) follows from Theorem~\ref{thm:the_main_theorem} together with the definition of representation stability (or, e.g., \cite[Theorem 1.13]{CEF}). Part (2)  follows for $\left(R_n/I_{\lambda_n}\right)^\vee$ from \cite[Theorem 1.5]{CEF} and for the dual representation $R_n/I_{\lambda_n}$ because the characters of complex representations of $S_n$ respect the operation of taking duals.   Part (3) is an easy corollary of Part (2) for $R_n/I_{\lambda_n}$ by~\cite[Theorem 1.5]{CEF} and for its dual because they have the same dimension as complex representations. 
\end{proof}


\section{Combinatorial consequences}\label{sec:combinatorial_consequences}

The results of Corollary~\ref{cor:springer_rep_is_representation_stable} open up new combinatorial questions about Springer fibers.  Representation stability guarantees that for any sequence of Young diagrams $\{\lambda_n\}$, the $k^{\mathrm{th}}$ degree of the cohomology of the corresponding Springer fibers stabilizes as a polynomial $p_k(n)$.  But what is this polynomial?  For instance, given the sequence $\lambda_1 \subseteq \lambda_2 \subseteq \lambda_3 \subseteq \cdots$ of Young diagrams:
 \begin{itemize}
     \item Is there an explicit formula for the dimensions $p_k(n)$ of the $k^{\mathrm{th}}$ graded part of $R_n/I_{\lambda_n}$?  What if the sequence contains a particular family of Young diagrams (e.g. hooks, two-row, two-column, etc.)?
     \item What is the minimal integer $s_k$ at which the dimension of $R_n/I_{\lambda_n}$ becomes polynomial?  Given $k$ and a sequence of Young diagrams, can we find some $s_k$ (not necessarily minimal) after which the dimension of $R_n/I_{\lambda_n}$ is polynomial?
     \item Can we show the Springer representations have polynomial dimension via the monomial bases $\mathcal{B}(\lambda_n)$ from Section~\ref{sec:Springer_theory}?
 \end{itemize}
 To the best of our knowledge, these are entirely new questions about Springer representations; they arise only because of the consequences of representation stability.  Moreover, our preliminary answers to these questions suggest additional combinatorial structures that may undergird Springer representations.  We give details and some concrete results in this section, followed in the next by (more) open questions.

\subsection{Minimal monomials}\label{subsec:min_mono}

Our first results suggest that the exponents that appear in the monomials within each $\mathcal{B}(\lambda)$ are more rigid than previously thought.  

\begin{definition}[Monomial type]
Let $x^{\alpha}$ be a monomial with exponent vector $(\alpha_1, \ldots, \alpha_n)$.  In other words, the exponent of the variable $x_i$ in $x^{\alpha}$ is $\alpha_i$.  Let $(\alpha_{i_1}, \ldots, \alpha_{i_r})$ be the nonzero exponents in the order in which they appear in $x^{\alpha}$.  This $r$-tuple is called the \textit{monomial type} of $x^{\alpha}$.
\end{definition}

For instance both $x_2x_3^2x_4$ and $x_3x_9^2x_{11}$ have monomial type $(1,2,1)$.

The following result says that if $x_I^{\alpha}$ is a monomial of a fixed monomial type in the GP-basis $\mathcal{B}(\lambda)$ then increasing the indices of the variables lexicographically while preserving the monomial type produces another monomial in the GP-basis $\mathcal{B}(\lambda)$.  

\begin{theorem}\label{thm:shift_the_indices}
Fix a composition $\alpha = (\alpha_1, \ldots, \alpha_k)$. For each ordered sequence $I=(i_1,\ldots,i_k)$  let $x_I^{\alpha}$ denote the monomial $\prod_{j=1}^k x_{i_j}^{\alpha_j}$.  If $x_I^{\alpha} \in \mathcal{B}(\lambda)$ and $I'$ satisfies 
\[I \leq I' \leq (n-k+1,\ldots,n-1,n)\] 
in lexicographic order then $x_{I'}^{\alpha} \in \mathcal{B}(\lambda)$.
\end{theorem}

\begin{proof}
First note that we may simply consider the case when $I$ and $I'$ differ by exactly one in exactly one entry.  Indeed, for each $j$ with $1 \leq j \leq k+1$, define the sequences $I_j = (i_1, i_2, \ldots, i_{j-1}, i'_j, i'_{j+1}, \ldots, i'_k)$.  Note that $I=I_{k+1}$ and $I' = I_1$.  Thus it suffices to prove that if $x_{I_{j+1}}^{\alpha} \in \mathcal{B}(\lambda)$ then $x_{I_{j}}^{\alpha} \in \mathcal{B}(\lambda)$.  Furthermore we may assume that $i'_j = i_j+1$ since repeating the argument would successively increment the index and imply the result for more general $i'_j$. Thus we prove the claim for $I_j$ and $I_{j+1}$ assuming that $i'_j = i_j+1$.  

The GP-algorithm proceeds the same for both monomials $x_{I_j}^{\alpha}$ and $x_{I_{j+1}}^{\alpha}$ in the variables $x_n, x_{n-1}, \ldots, x_{i_j+2}$.  Let $\lambda^{(j+2)}$ be the Young diagram left after those steps.  Implementing the GP-algorithm for $x_{I_{j+1}}^{\alpha}$ removes a box from the first row and then from the $\alpha_j+1^{\mathrm{th}}$ row of $\lambda^{(j+2)}$, whereas for $x_{I_j}^{\alpha}$ it removes a box from the $\alpha_j+1^{\mathrm{th}}$ row and then from the first row.  Call the resulting Young diagrams $\lambda^{(j+1)}$ and $\lambda^{(j)}$ respectively.

We now show that $\lambda^{(j+1)} \unrhd \lambda^{(j)}$ in dominance order.  If the first row of $\lambda^{(j+2)}$ is strictly larger than the $\alpha_j+1^{\mathrm{th}}$ row of $\lambda^{(j+2)}$ then the result is clear, since then $\lambda^{(j+1)} = \lambda^{(j)}$.  If not, then removing a box from the first row of $\lambda^{(j+2)}$ leaves a Young diagram whose shape is constrained: at least the first $\alpha_j$ rows have the same length and there is at least one row immediately after of length one less.  Thus $\lambda^{(j+1)} \unrhd \lambda^{(j)}$ with equality in the case that exactly the first $\alpha_j+1$ rows of $\lambda^{(j+2)}$ are the same length, and any subsequent rows are smaller.

By Proposition~\ref{proposition: dominance order}, we have $\mathcal{B}(\lambda^{(j+1)}) \subseteq \mathcal{B}(\lambda^{(j)})$.  Since $x_{i_1}^{\alpha_1}x_{i_2}^{\alpha_2} \cdots x_{i_{j-1}}^{\alpha_{j-1}}$ is a GP-monomial in $\mathcal{B}(\lambda^{(j+1)})$ it is also a GP-monomial in $\mathcal{B}(\lambda^{(j)})$ so the monomial $x_{I_j}^{\alpha}$ can be obtained from the GP-algorithm starting with the shape $\lambda^{(j+2)}$.  Thus $x_{I_j}^{\alpha}$ is in the GP-basis for $\mathcal{B}(\lambda)$ as desired, and the claim follows.
\end{proof}

In fact we conjecture that there is a unique minimal monomial of each monomial type, in the following lexicographic sense.

\begin{definition}[Minimal monomials]
Fix a monomial type $\alpha$ and a partition $\lambda$ of $n$.  The monomial $x_I^{\alpha} \in \mathcal{B}(\lambda)$ is a {\em minimal monomial} of type $\alpha$ if for any other $x_{I'}^{\alpha} \in \mathcal{B}(\lambda)$ of the same monomial type, the index sets satisfy $I \leq I'$ in lexicographic order.
\end{definition}

The next result proves that when $\lambda$ has exactly two rows then there is a unique minimal monomial of each monomial type in $\mathcal{B}(\lambda)$.  The previous result then says that this minimal monomial produces all other monomials of that fixed monomial type within $\mathcal{B}(\lambda)$.

\begin{theorem}\label{thm:2_row_case_min_monomials}
Fix a two-row Young diagram $\lambda = (n-k,k)$. Then the basis $\mathcal{B}(\lambda)$ of $R_{n}/I_\lambda$ has the unique minimal monomial set $\left\{ \prod\limits_{j=1}^i x_{2j} \right\}_{i=1}^k$.
\end{theorem}

\begin{proof}
Since $\lambda$ has two rows, the maximum exponent that appears in a monomial for $\lambda$ is $1$.  Since $\lambda$ has $k$ boxes in the second row, the maximum degree in a monomial for $\lambda$ is $k$.  So every monomial type that appears in $\mathcal{B}(\lambda)$ is the partition given by $j$ copies of $1$, where $j$ is any integer with $0 \leq j \leq k$.  

We claim that the minimal monomial of degree $j$ is $x_2 x_4 \cdots x_{2j}$. To see this, let $x_I \in \mathcal{B}(\lambda)$ have degree $j$ and index set $I = (i_1, \ldots, i_j)$. It suffices to show that $(2, 4, \ldots, 2j) \leq I$ in lexicographic order, namely that $i_r \geq 2r$ for each $r \in \{1,\ldots,j\}$. Recall that in the GP-algorithm, exactly one rightmost box in a row is removed when going down each level of the GP-tree. Going from Level $i$ to Level $i-1$, the box removed is in the first row of the Level $i$ diagram if $i \notin \{i_1, \ldots, i_j\}$ and the second row otherwise. After a box is removed from the first row, the rows of the resulting Level $i-1$ diagram are switched if necessary to maintain the shape of a Young diagram. 

Now suppose $r \in \{1,\ldots,j\}$. The path in the GP-tree from Level $i_r$ to Level $i_r - 1$ removes a box from the second row of the Level $i_r$ diagram $\lambda^{(i_r)}$. The diagram $\lambda^{(i_r)}$ must have at least $r$ boxes in the second row since the $r-1$ factors $x_{i_{r-1}}, x_{i_{r-2}}, \ldots, x_{i_1}$ remaining in $x_I$ correspond to $r-1$ more second-row boxes removed in the remaining $i_r - 2$ levels of the GP-tree below Level $i_r - 1$. Since all Level $i_r$ diagrams have $i_r$ boxes and $\lambda^{(i_r)}$ has at least $r$ boxes in its second row, we conclude $i_r \geq 2r$ as desired.
\end{proof}

We conjecture that for each $\mathcal{B}(\lambda)$ and each monomial type that appears in $\mathcal{B}(\lambda)$, there is a unique minimal monomial of that monomial type (see Section~\ref{subsec:unique_min_mono}).

\subsection{Polynomial dimension when the number of rows exceed a certain minimum}\label{subsec:poly_dim}

We can explicitly compute the polynomial dimensions described in Corollary~\ref{cor:springer_rep_is_representation_stable} in some cases, including when the Young diagram has ``enough" rows.  In this case, the dimensions in low degree coincide with the corresponding dimensions for the diagonal coinvariant algebra, as described in the next theorem.  Kim's version of stability coincides with ours in this case \cite[Remark in Section 4]{Kim_Stability}.

\begin{theorem}\label{thm: at least k+1 rows gives a polynomial dimension formula}
Let $\dim_i(R_n/I_{\lambda_n})$ denote the dimension of the $i^{\mathrm{th}}$-degree part of $R_n/I_{\lambda_n}$.  Let $\{ \lambda_n \}$ be a nested sequence of Young diagrams such that $|\lambda_n| = n$ for all $n$.  If for some $N$ the diagram $\lambda_N$ contains at least $k+1$ rows, then $\dim_i\left(\bigslant{R_n}{I_{\lambda_n}}\right)$ agrees with the dimension of the $i^{\mathrm{th}}$-degree part of the diagonal coinvariant algebra $\bigslant{R_n}{\left\langle e_j \right\rangle_{j=1}^n}$ for all $i \leq k$ and $n \geq N$.

In particular  for all $i \leq k$ and $n \geq N$ we have $\dim_i\left(\bigslant{R_n}{I_{\lambda_n}}\right) = p_i(n)$  where $p_i(n)$ is the polynomial that gives the number of permutations of the set $\{1,2,\ldots,n\}$ with exactly $i$ inversions.
\end{theorem}

\begin{proof}
Let  $\col_n = (1,\ldots,1)$ denote the column partition with exactly $n$ parts.  The diagonal coinvariant algebra is $R_n/I_{\col_n}$.  There is a unique monomial $x_2 x_3^2 x_4^3 \cdots x_n^{n-1} \in \mathcal{B}(\col_n)$ of maximal degree; it is obtained by choosing the far-right edge at each level in the GP-tree.  Garsia-Procesi proved that their basis elements form a lower-order ideal with respect to division \cite[Proposition 4.2]{GarPro92}.  It follows that the other elements of $\mathcal{B}(\col_n)$ are precisely the divisors of this maximal monomial, so
\begin{equation}\label{basis for column partition of size n}
\mathcal{B}(\col_n) = \{ x_2^{\alpha_2} x_3^{\alpha_3} x_4^{\alpha_4} \cdots x_n^{\alpha_n} \; | \; 0 \leq \alpha_j \leq j-1 \}.
\end{equation}

Now suppose that $\lambda_n$ has at least $k+1$ rows.  Then $\lambda_n \supseteq \col_{k+1}$ and so $\mathcal{B}(\lambda_n) \supseteq \mathcal{B}(\col_{k+1})$ by Lemma~\ref{lemma: containment and Garsia-Procesi basis}. By Theorem~\ref{thm:shift_the_indices} we can lexicographically increase the subscripts of each monomial in $\mathcal{B}(\col_{k+1})$ to produce monomials $x_I^{\alpha}$ for each $k$-element subset $I \subseteq \{x_2, \ldots, x_n\}$ and exponents $\alpha$ which satisfy $0 \leq \alpha_j \leq j-1$ for each $j=1,2,\ldots,k+1$.   This proves that $\mathcal{B}(\lambda_n)$ also contains all of those possible elements of degrees $0, 1, \ldots, k$ in the variable set $\{x_2, \ldots, x_n\}$.  Trivially, these are the same as the monomials of degrees $0, 1, \ldots, k$ in the basis for $\mathcal{B}(\col_n)$ above, so in these degrees the basis elements for  $R_n/I_{\lambda_n}$ and the diagonal coinvariant algebra $R_n / I_{\col_n}$ coincide.  The claim about $p_i(n)$ follows from the similar statement for the $i^{\mathrm{th}}$ graded part of $R_n / I_{\col_n}$.
\end{proof}

The polynomial function $p_i(n)$ can be found explicitly as an alternating sum of certain combinations.   Example~\ref{exam:poly_dim_closed_form} gives closed formulas for $p_i(n)$ for the first few values of $i$.

\begin{remark}
Since the diagonal coinvariant algebra is the cohomology ring of the full flag variety, this result implies that the Garsia-Procesi bases for the cohomology of the Springer fibers coincide with the cohomology of the full flag variety in many degrees.  More precisely if $\mathcal{S}_{\lambda_n}$ denotes the Springer fiber corresponding to the nilpotent of Jordan type $\lambda_n$ then we have an isomorphism $H^i(\mathcal{S}_{\lambda_n}) \cong H^*(GL_n(\mathbb{C})/B)$ for all $i \leq k$ and $n \geq N$, where $k$ and $N$ are as given in Theorem~\ref{thm: at least k+1 rows gives a polynomial dimension formula}.
\end{remark}

\begin{example}
Below we give the basis elements for  $\mathcal{B}(\col_5)$ in degrees $0$ through $3$ and the cardinalities of the $k^{\mathrm{th}}$ degree parts for $4 \leq k \leq 10$, using the unimodality of the sequence of dimensions for $k\geq 6$.
\begin{center}
\begin{tabular}{c|c|l}
degree & \# & \hspace{.75in} monomials in $\mathcal{B}(\col_5)$ \\ \hline \hline
0 & 1 & 1 \\ \hline
1 & 4 & $x_i$ for $2 \leq i \leq 5$ \\ \hline
2 & 9 & $x_i^2$ for $3 \leq i \leq 5$\\
& & and $x_ix_j$ for $2 \leq i<j \leq 5$\\ \hline
3 & 15 & $x_i x_j x_k$ for $2 \leq i<j<k \leq 5$,\\
  &    & $x_i x_j^2$ for $2 \leq i < j \leq 5$,\\
  &    & $x_i^2 x_j$ for $3 \leq i < j \leq 5$, and\\
	&		 & $x_i^3$ for $4 \leq i \leq 5$\\  \hline
4 & 20 & \\ \hline
5 & 22 & \\ \hline
$6 \leq k \leq 10$ & same cardinality as & \\
  & $(10-k)^{\mathrm{th}} \mbox{ degree part}$ & \\ \hline
$k > 10$ & 0 & \\
\hline
\end{tabular}
\end{center}
Compare this to the basis $\mathcal{B}(\lambda)$ for the partition $\lambda = (2,2,1)$ that we computed in Example~\ref{example: GP-tree}.  The basis elements for $\mathcal{B}(\lambda)$ and $\mathcal{B}(\col_n)$ coincide in degrees 0, 1, and 2, as expected.  However in degree 3 the set $\mathcal{B}(\col_5)$ contains  four basis elements that are not in $\mathcal{B}(\lambda)$, namely $x_2 x_3 x_4$, $x_3^2 x_4$, $x_4^3$, and $x_5^3$.
\end{example}

\begin{example}\label{exam:poly_dim_closed_form}
We give the following closed formulas for the dimensions $p_i(n)$ from Theorem~\ref{thm: at least k+1 rows gives a polynomial dimension formula} when $0 \leq i \leq 6$:
\begin{itemize}
\item  When $i=0$ we have $p_0(n) = 1$.  
\item When $i = 1, 2, 3, 4$ we have
\begin{equation}\label{nice_closed_formula}
p_i(n) = \binom{n-2+i}{i} - \binom{n-3+i}{i-2}.
\end{equation}
\item When $i=5$ we have
\[p_5(n) = \binom{n+3}{5} - \binom{n+2}{3} + 1.\]
\item When $i=6$ we have
\[p_6(n) = \binom{n+4}{6} - \binom{n+3}{4} + n.\]
\end{itemize}

To prove these, let $i \in \{1,2,3,4\}$.  Our convention is that $0 \in \mathbb{N}$. Define the set
\begin{align*}
\widetilde{E_i} &= \left\{ (d_2,\ldots,d_n)\in \mathbb{N}^{n-1} \; \middle\vert \; \sum_{k=2}^n d_k = i \mbox{ and } d_k \leq i \right\}
\end{align*}
of nonnegative integer solutions to the equation $d_2 + \cdots + d_n = i$.  We know $|\widetilde{E_i}| = \binom{(n-2)+i}{i}$.  Define $E_i \subseteq \widetilde{E_i}$ to be the subset satisfying also $d_k \leq k-1$ for each $k$ so that $p_i(n) = |E_i|$.  We now compute $|E_i| = |\widetilde{E_i}| -  |E_i^c|$.  

For each $m$, define the set $F_{m,i} \subseteq \widetilde{E_i}$ by the condition that $d_m > m-1$. Observe that $E_i^c = \textstyle\bigcup\limits_{m=2}^{n} F_{m,i}$.  Moreover the sets $F_{m,i}$ are pairwise disjoint. Indeed, suppose there were $m_1<m_2$ with an element $(d_2,\ldots,d_n) \in F_{m_1,i} \cap F_{m_2,i}$.  Then $d_{m_1} > m_1-1 \geq 1$ and $d_{m_2} > m_2-1 \geq 2$ and so $d_2 + \cdots + d_n \geq 5$.  This contradicts $i \in \{1,2,3,4\}$.  

So $|E_i^c| = \textstyle\sum\limits_{m=2}^n|F_{m,i}|$.  We now find $|F_{m,i}|$.  Since $d_m > m-1$ we write $d_m = m + d_m'$ for some $d_m' \geq 0$ and then count instead the solutions to $d_2 + \cdots + d_m' + \cdots + d_n = i-m$. Thus
\begin{align*}
|E_i^c| = \sum_{m=2}^n |F_{m,i}| = \sum_{m=2}^n \binom{(n-2) + (i-m)}{i-m}.
\end{align*}
As long as $i \leq n$ this sum has the form $|E_i^c| = \cdots + \binom{r_m + 2}{2} + \binom{r_m + 1}{1} + \binom{r_m}{0}$ for some $r_m$.  Using a combinatorial identity we write $|E_i^c| = \binom{(n-2) + (i-2) + 1}{i-2}$.  This gives the formula for $p_i(n)$ in Equation~(\ref{nice_closed_formula}).

When $i=5$ and $i=6$ we can use essentially the previous argument, since only a few pairs of $F_{m,i}$ share elements.  When $i=5$ the sets $F_{2,5}$ and $F_{3,5}$ both contain $(2,3,0,0,0,\ldots,0)$.  Inclusion-Exclusion gives the desired formula.  Similarly if $i=6$ then  $F_{2,6} \cap F_{3,6}$ contains the two elements $(3,3,0,0,\ldots,0)$ and $(2,4,0,0,\ldots,0)$ as well as the $n-3$ elements $(2,3,d_4,\ldots,d_n)$ where exactly one of the remaining entries $d_k$ is nonzero (and thus is 1).  Additionally  $F_{2,6} \cap F_{4,6}$ contains the unique element $(2,0,4,0,0,\ldots,0)$.  Inclusion-Exclusion gives the desired formula.  

This process becomes more complicated when $i$ is larger.
\end{example}

Kim gives equivalent but different versions of the formulas in Example~\ref{exam:poly_dim_closed_form}~\cite[Example~3.3]{Kim_Stability}.


\section{Open questions}\label{sec:open_questions}

This section describes a set of questions that arise from our analysis of representation stability, including how to characterize the polynomials into which the dimensions stabilize, and how to describe the monomials in $\mathcal{B}(\lambda)$.

\subsection{Minimal monomials and monomial types}\label{subsec:unique_min_mono}

In Subsection~\ref{subsec:min_mono} we discussed monomial types and minimal monomials.  We pose several questions here.

\begin{question}
For each Young diagram $\lambda$ and monomial type $\alpha$, is there a unique monomial $x_I^{\alpha} \in \mathcal{B}(\lambda)$ that is minimal with respect to lexicographic ordering, in the sense that if $x_{I'}^{\alpha} \in \mathcal{B}(\lambda)$ then $I \leq I'$?
\end{question}

All examples that we have computed do in fact have minimal monomials.

There are several related questions about how the set of monomial types that appear in $\mathcal{B}(\lambda)$ are related to $\lambda$. The first essentially asks how to predict the monomial types that appear for a given $\lambda$, while the second asks which $\mathcal{B}(\lambda)$ contain a given monomial type $\alpha$.

\begin{question}
For each Young diagram $\alpha$, let $\mathcal{A}(\lambda)$ be the set of monomial types in $\mathcal{B}(\lambda)$.  Is there a quick algorithm to construct $\mathcal{A}(\lambda)$ for arbitrary $\lambda$? 
\end{question}

We know there is an algorithm to construct $\mathcal{A}(\lambda)$, namely use the GP-algorithm to find all of $\mathcal{B}(\lambda)$ and then identify which monomial types appear.  The previous question is asking for a direct algorithm, or for characterizations of the monomial types that do or do not appear.

\begin{question}
Fix a monomial type $\alpha$.  For which $\lambda$ is there a monomial of type $\alpha$ in $\mathcal{B}(\lambda)$?
\end{question}

For instance $\mathcal{B}(\lambda)$ contains a monomial of type $(k)$ if and only if $\lambda$ has at least $k+1$ rows.  

\subsection{Polynomial dimension for particular families of Young diagrams} 
In this section, we return to questions from Subsection~\ref{subsec:poly_dim} to ask for concrete formulas for dimensions of the Springer representations, and especially for the polynomials $p_k(n)$ to which they stabilize.

Representation stability relies on sequences of diagrams while the literature on Springer fibers instead uses families of Young diagrams like two-row, two-column, hook, and so on.  To study families of Young diagrams, we choose a sequence $\{\lambda_n\}$ that characterizes particular shapes, as Example \ref{example: 2-column} does for two-column Young diagrams.  

\begin{question}
For what families of Young diagrams $\{\lambda_n\}$ can we find an explicit closed formula for the polynomial dimension $p_k(n)$?
\end{question}
 
For instance, Kim gave a formula for two-row Young diagrams~\cite{Kim_Euler_char}, following ideas due to Fresse~\cite[Example~4.5]{Fresse}. 

When we can identify minimal monomials (unique or not), strictly combinatorial techniques might then give information about the polynomial dimension formula.  More precisely:

\begin{question}
Can we find minimal monomials for other shapes, and use them together with the index-incrementation tools in Theorem~\ref{thm:shift_the_indices} to describe the dimension polynomials, either partially or completely?
\end{question}

For instance, Theorem~\ref{thm:2_row_case_min_monomials} gave the minimal monomials for the case of the two-row diagrams $\lambda = (n-k,k)$.  In that case, the previous question sketches an alternate proof of Kim's result for two-row Young diagrams.

\subsection{Stable shapes} 

The central question of this section is to determine the integer $s_k$ after which the polynomial $p_k(n)$ counts the degree-$k$ monomials in $\mathcal{B}(\lambda_n)$. More colloquially, when does the dimension stabilize?  

We conjecture that there is a core shape such that as soon as $\lambda_n$ contains the core shape, the dimension stabilizes.  The following describes one possible choice of a core shape, though we do not believe that it is in general minimal.

\begin{definition}
Let $\lambda_1 \subseteq \lambda_2 \subseteq \cdots$ be a sequence of Young diagrams with $|\lambda_n| = n$ for each $n$.  Denote by $\tau_{k+1}$ the staircase partition $(k+1, k, \ldots,1)$.  The \textit{$k$-stable shape} of $\{ \lambda_n \}$ is the partition $(\cup_n \, \lambda_n) \cap \tau_{k+1}$.
\end{definition}

\begin{example}\label{example: 2-column}
Let $\{ \lambda_n \}$ be the nested sequence of Young diagrams defined by $\lambda_1 = (1)$, $\lambda_{2n} = (\underbrace{2, 2, \ldots, 2}_{n \; \mbox{\tiny times}})$, and $\lambda_{2n+1} = (\underbrace{2, 2, \ldots, 2}_{n \; \mbox{\tiny times}},1)$.  Then $\{ \lambda_n \}$ is the following
$$\mbox{\small $\Yvcentermath1 \yng(1) \rightarrow \yng(2) \rightarrow \yng(2,1) \rightarrow \yng(2,2) \rightarrow \yng(2,2,1) \rightarrow \yng(2,2,2) \rightarrow \yng(2,2,2,1) \rightarrow \yng(2,2,2,2) \rightarrow \cdots$}$$
The $1$-, $2$- and $3$-stable shapes are $\mbox{\tiny \Yvcentermath1 \yng(2,1)}$, $\mbox{\tiny \Yvcentermath1 \yng(2,2,1)}$ and $\mbox{\tiny \Yvcentermath1 \yng(2,2,2,1)}$, respectively.
\end{example}

\begin{conjecture}
Let $\{ \lambda_i \}$ be a sequence of Young diagrams such that $\lambda_1 \subseteq \lambda_2 \subseteq \cdots$ with $|\lambda_i| = i$.  Let $N$ be the smallest integer for which $\lambda_N$ contains the $k$-stable shape of $\{ \lambda_i \}$.  Then there exists a polynomial $p_k(n)$ that gives the dimension of the $k^{\mathrm{th}}$ graded part of the Springer representation $R_n/I_{\lambda_n}$ for all $n \geq N$.
\end{conjecture}

If true, this conjecture would provide more tools to explicitly compute  $p_k(n)$.

\begin{acknowledgements}
The authors thank Tom Church, Jordan Ellenberg, and Benson Farb for useful conversations. We also thank the anonymous referees for their helpful comments and suggestions.
\end{acknowledgements}



\end{document}